\documentclass[12pt]{article}
\usepackage{latexsym,amsmath,theorem,graphicx,amssymb,a4wide}
\usepackage[utf8]{inputenc}
\usepackage{hyperref}

\numberwithin{equation}{section}

\newtheorem{teo}{Theorem}[section]
\newtheorem{lemma}[teo]{Lemma}
\newtheorem{pro}[teo]{Proposition}

\newtheorem{de}[teo]{Definition}
\newtheorem{con}[teo]{Condition}

{\theorembodyfont{\rmfamily}\newtheorem{re}[teo]{Remark}}
{\theorembodyfont{\rmfamily}}

\DeclareMathOperator{\EE}{\mathbb{E}}
\DeclareMathOperator{\PP}{\mathbb{P}}
\DeclareMathOperator{\Var}{Var} \DeclareMathOperator{\Cov}{Cov}

\newcommand{\ind}{\mathbf{1}}

\newcommand{\real}{\mathbb{R}}

\newcommand{\integer}{\mathbb{Z}}
\renewcommand{\natural}{\mathbb{N}}

\newcommand{\as}{\emph{a.s. }}
\newcommand{\ie}{\emph{i.e. }}
\newcommand{\iid}{\emph{i.i.d. }}
\newcommand{\eg}{\emph{e.g. }}

\renewcommand{\epsilon}{\varepsilon}

 \title{\bf Propagation of chaos for a General Balls into Bins dynamics \thanks{The present work was financially supported  by  PRIN 20155PAWZB ``Large Scale Random Structures''. }}

\author{Nicoletta Cancrini\footnote{nicoletta.cancrini@univaq.it, DIIIE Università dell'Aquila, L'Aquila Italy.}
  \ and Gustavo Posta\footnote{gustavo.posta@uniroma1.it, Dipartimento di Matematica ``G. Castelnuovo'', Università di Roma ``la Sapienza'', Roma Italy.}
}

\begin{document}

\date{}

\noindent

\maketitle
\thispagestyle{empty}

\begin{abstract}
  Consider $N$ balls initially placed in $L$ bins.
  At each time step take a ball from each non-empty bin and \emph{randomly} reassign \emph{all} the balls into the bins.
  We call this finite Markov chain \emph{General Repeated Balls into Bins} process.
  It is a discrete time conservative interacting particles system with parallel updates.
  Assuming a \emph{quantitative} chaotic condition on the reassignment rule we prove a \emph{quantitative} propagation of chaos for this model.
  We furthermore study some equilibrium properties of the limiting nonlinear process.
\end{abstract}

\section{Introduction}
\label{sec:intro}

In recent years interacting stochastic processes with parallel updates has received an increasing interest in the scientific literature and particularly in the probabilistic one.
Important applications are the \emph{neuronal networks} models see \cite{Ga:Lo}, \cite{DM:Ga:Lo:Pr} and reference therein, \emph{networks of interacting CPUs}, see for example  \cite{Be:Cl:Na:Pa:Po} and \emph{probabilistic cellular automata} dynamics, see \cite{Lou:Na} for an overview.
Parallel updating rules introduce new difficulties in the study of these processes in particular 
in many cases the dynamics is not reversible and the invariant measure is unknown.
This means that the description of the system at equilibrium 
and  the behaviour of a huge number of interacting components is not generally available.
For this reason many authors studied the large scale limit of these systems in the so called \emph{propagation of chaos} framework (see \cite{An:DP:Fi},\cite{Ca:Po} and \cite{Fo:Lo}). Propagation of chaos gives the link between the microscopic and the macroscopic level, in particular it says that the system, in the large scale limit, behaves as the components are independent and each one follows a \emph{non linear dynamics}.
Propagation of chaos is a largely studied subject which dates back to the seminal paper of M. Kac \cite{Ka}, see  \cite{Sz} for an introduction. However, as noted by \cite{Ga:Lo}, it is not clear at all that it holds for systems with parallel updates and only few examples are known to exhibit this property. The fact that parallel jumps may interfere with asymptotic independence makes propagation of chaos for these models an interesting field.
Quantitative estimates on the difference between the original interacting process, when the number of components is large, and the non linear dynamics of the infinite volume limit are needed to evaluate the approximation error.
In \cite{An:DP:Fi} this problem is considered in the $L^1$ framework using the Wasserstein-one distance for systems of interacting diffusions with simultaneous jumps and the authors prove propagation of chaos for a wide class of models. We here study the behaviour of the total variation distance for a  \emph{General Repeated Balls into Bins} process which is not contained in this class and new ideas and techniques, accounting the different features of these processes, are needed.

To define our model consider $N$ balls initially placed in $L$ bins. 
Take a ball from each non-empty bin and \emph{randomly} reassign the balls into the bins, then
iterate independently this procedure at each time step.
The random evolution of the number of balls in each bin is an ergodic finite state Markov chain that we call \emph{General Repeated Balls-into-Bins}  (GRBB) process.
A particular case of the GRBB process is the \emph{Repeated Balls-into-Bins} (RBB) process, in which the balls are uniformly and independently reassigned into bins.
For the RBB process in \cite{Ca:Po} we proved the propagation of chaos in the weak limit sense and without any quantitative estimate.
For the GRBB process the random reassignment has a \emph{general} distribution. 
The systems in this class are conservative interacting particles systems, in discrete time, with parallel updates.
Usually the GRBB process is not reversible and its invariant measure cannot always be calculated.

The GRBB process, as the RBB process, appears naturally in different applicative contexts. 
For example we can think to balls in every bin as customers or tasks in a queue.
Customers are served at discrete times and each served customer is reassigned to a random queue.
In this setting the GRBB process is a discrete time \emph{closed Jackson network} \cite{Ja,Ke}.
The parallel updating is justified, see for example \cite{Be:Cl:Na:Pa:Po}, by thinking to customers as tasks (or \emph{tokens}) in a network of parallel CPUs which are reassigned at every round.

As noted above  we are interested in the behaviour of the GRBB process for large $L$. In physics this large scale limit is known as the \emph{thermodynamic limit}.
Assuming a \emph{quantitative} chaotic condition on the reassignment rule, see Condition~\ref{con:1}, we prove a \emph{quantitative} propagation of chaos.
Quantitative here means that we give the explicit rate of convergence of the empirical measure of the GRBB process to the nonlinear process distribution as $L\to+\infty$, see Theorem~\ref{teo:mai}. We obtain the same rate of \cite{An:DP:Fi} for systems of interacting diffusions with simultaneous jumps. 
The quantitative chaotic condition on the reassignment rule is strong but natural as can be seen as the distance between a \emph{canonical} and the corresponding \emph{grand canonical} measure.
This problem is known in literature as \emph{equivalence of ensembles} in the thermodynamic limit
(see for example \cite{Ca:Ma}).
The existence of this strong quantitative equivalence of ensembles depends on the model and can be a difficult problem. 
In this paper we prove that this condition holds for some natural applications of our model when the number of particles is  proportional to $L$ which is the interesting case in the thermodynamic  limit.
We remark that, as in \cite{An:DP:Fi}, our propagation of chaos result holds in finite time intervals.
However, to our knowledge, uniform in time propagation of chaos for models with parallel updates is actually an unexplored field and except that for some trivial cases, we cannot say anything on the propagation of chaos for the stationary measure.

The paper is organized as follows.
In Section~\ref{sec:cbp} we define the GRBB process and the nonlinear process, we then prove the theorem on quantitative propagation of chaos.
In Section~\ref{sec:app} we apply Theorem~\ref{teo:mai} to three cases of the GRBB process depending on different choices of the reassignment rule: Fermi-Dirac, Maxwell-Boltzmann and Bose-Einstein statistics.
The results of this section are mainly obtained by coupling arguments.
We give here also a mixing time bound for the process with Fermi-Dirac reassignment rule, while the analogous problem for the Maxwell-Boltzmann reassignment rule is contained in \cite{Ca:Po:1}.
In the last section we study the long time behavior of the nonlinear process. 

\section{Construction and main result}
\label{sec:cbp}

We denote with $\integer_+$ the set of the non-negative integers and define $\natural:=\integer_+\setminus\{0\}$. 
For any denumerable set $S$ we denote with $|S|$ its cardinality.
Furthermore we denote with $\mathcal{P}(S)$ the metric space of probability measures on $S$ endowed with the \emph{total variation distance}
\begin{equation*}
  \Vert p-q\Vert
  :=\sup_{A\subseteq S}\big(p(A)-q(A)\big)
  =\frac{1}{2}\sum_{s\in S}\vert p(\{s\})-q(\{s\})\vert.
\end{equation*}
We define the \emph{empirical measure function} $\rho_L\colon\integer_+^L\to\mathcal{P}(\integer_+)$ as
\begin{equation*}
  \rho_L(\xi)
  :=\frac{1}{L}\sum_{x=1}^L\delta_{\xi_x},
\end{equation*}
where $\delta_n$ is the Dirac mass at $n\in\integer_+$, and the map $w^L\colon\integer_+^L\to\integer_+^L$ as
\begin{equation*}
  w^L(\xi)
  :=(\ind(\xi_1>0),\dots,\ind(\xi_L>0)).
\end{equation*}
To keep notation simple in the following we denote with the term \emph{constant} a positive number which does not depend on $L$.
Furthermore, when this does not cause confusion,  we use the same symbol to denote different constants.

\subsection{Propagation of chaos of the GRBB process}
\label{sec:grbb}

We here define the GRBB process and its corresponding \emph{nonlinear} process, which are discrete time Markov chains.

\begin{de}
  \label{de:grb}
  The GRBB process $(\eta^L(t))_{t\geq0}$ is the discrete time Markov chain with values in $\integer_+^L$ defined as follows.
  Assume that, for some $t\in\integer_+$,  $\eta^L(t)=\xi\in\integer_+^L$ and $q:=\rho_L(\xi)$ then 
  \begin{equation}
    \label{eq:grb}
    \eta^L(t+1)=\xi-w^L(\xi)+B^{L,q}.
  \end{equation}
  $B^{L,q}$ is a random vector with values in $\integer_+^L$.
  It is independent from everything, independently generated at each time step $t$ and satisfies
  \begin{equation}
  \label{eq:co}
  \sum_{x=1}^L B_x^{L,q}=(1-q(\{0\}))L.
\end{equation}
\end{de}

Last equation assures the conservation of the number of particles for the GRBB process.

\begin{de}
  \label{de:nlp}
  Given a measurable map $\psi\colon\mathcal{P}(\integer_+)\to\mathcal{P}(\integer_+)$ we define the $\psi$-\emph{nonlinear process $(\eta(t))_{t\geq0}$} as the discrete time random process with values in $\integer_+$ defined as follows.
  For some $t\in\integer_+$ let $q\in\mathcal{P}(\integer_+)$ be the distribution of $\eta(t)$ and assume $\eta(t)=\eta\in\integer$ then
\begin{equation}
  \label{eq:nl1}
  \eta(t+1)=\eta-\ind(\eta>0)+B^q.
\end{equation}
 $B^q$ is a random variable with values in $\integer_+$, it has distribution $\psi(q)$, it is independent from everything and it is independently generated at each time step $t$.
\end{de}

We want to provide a quantitative estimate on the rate of convergence of the empirical measure of the GRBB process to the distribution of a \emph{corresponding} nonlinear process.
The following is a sufficient condition to this aim.
\begin{con}
  \label{con:1}
  The distribution of $B^{L,q}$ is symmetric.
  For any $q\in\mathcal{P}(\integer_+)$ denote with $\nu_L^q\in\mathcal{P}(\integer_+^2)$ the distribution of $(B_1^{L,q},B_2^{L,q})$ and assume that there exists $\mu^q\in\mathcal{P}(\integer_+)$ and a constant $C$ such that
  \begin{equation}
    \label{eq:co1}
    \sup_{q\in\mathcal{P}(\integer_+)}\Vert\nu_L^q-\mu^q\otimes\mu^q\Vert
    \leq \frac{C}{L}.
  \end{equation}
\end{con}

  In the context of Gibbs measures the above condition is rather natural.
It can be interpreted as an equivalence of ensembles estimate because it  states that the distance between the two sites marginal of the \emph{canonical} and \emph{grand canonical} ensemble decreases as the inverse of the volume $L$.

Among the $\psi$-nonlinear processes we need to choose the one that gives the limiting evolution of the GRBB process.
This is done in the following definition.
\begin{de}
  Given a GRBB process, such that the random vector $B^{L,q}$ satisfies Condition~\ref{con:1}, the \emph{corresponding} nonlinear process is the $\psi$-nonlinear process with $\psi(q):=\mu^q$.  
\end{de}

We can now state our main result on propagation of chaos.

\begin{teo}
  \label{teo:mai}
  Let $(\eta^L(t))_{t\geq0}$ be a GRBB process with symmetric initial distribution and such that $\sup_L\EE(\eta_1^L(0))<+\infty$.
  Assume that Condition~\ref{con:1} holds.
  Let $(\eta(t))_{t\geq0}$ be the corresponding nonlinear process and assume that $\psi\colon \mathcal{P}(\integer_+)\to \mathcal{P}(\integer_+)$ is Lipschitz.
  Define $Q_L(t):=\rho_L(\eta^L(t))$ and let $Q(t)$ be the distribution of $\eta(t)$.
  If there exists a constant $C$ such that 
  \begin{equation}
    \label{eq:ma1}
    \PP\Big(\Vert Q_L(0)-Q(0)\Vert>\delta\Big)
    \leq\frac{C}{\sqrt{L}},
  \end{equation}
  then there exists a constant $C'$ such that
  \begin{equation}
        \label{eq:ma2}
    \PP\Big(\sup_{t\in[0,T]}\Vert Q_L(t)-Q(t)\Vert>\delta\Big)
    \leq\frac{C'}{\sqrt{L}}.
  \end{equation}
\end{teo}

\begin{proof}
  First of all observe that
    \begin{multline*}
    \PP\Big(\sup_{t\in[0,T]}\Vert Q_L(t)-Q(t)\Vert>\delta\Big)
    =\PP\big(\exists\, t\in[0,T] \colon \Vert Q_L(t)-Q(t)\Vert>\delta\big)\\
    \leq\sum_{t=0}^T \PP\big(\Vert Q_L(t)-Q(t)\Vert>\delta\big)
    \leq (T+1) \sup_{t\in[0,T]}\PP\big(\Vert Q_L(t)-Q(t)\Vert>\delta\big).
  \end{multline*}
  Thus it is enough to show that for any $t\in[0,T]$ there exists a constant $C$ such that
  \begin{equation}
    \label{eq:qt}
    \PP\big(\Vert Q_L(t)-Q(t)\Vert>\delta\big)
    \leq \frac{C}{\sqrt{L}}.
  \end{equation}  
  We prove \eqref{eq:qt} by induction on $t$.
  By hypothesis \eqref{eq:qt} is true for $t=0$. 
  We assume it holds for some $t\geq0$ and prove it for $t+1$.

  Observe that if $Q(t)=q\in\mathcal{P}(\integer_+)$ then, for any $n\in\integer_+$
  \begin{equation}
    \label{eq:qn}
    Q(t+1)(\{n\})
    =q(\{0\})\PP(B^q=n)+\sum_{k=1}^{n+1}q(\{k\})\PP(B^q=n-k+1)
    :=F(q)(\{n\}),
  \end{equation}
  where $F\colon\mathcal{P}(\integer_+)\to\mathcal{P}(\integer_+)$.
  $F$ is Lipschitz because $\psi$ is Lipschitz.

  Adding and subtracting terms we have that
  \begin{multline*}
    \Vert Q_L(t+1)-Q(t+1)\Vert
    \leq \Vert Q_L(t+1)-\EE[Q_L(t+1)\vert Q_L(t)] \Vert\\
    +\Vert \EE[Q_L(t+1)\vert Q_L(t)]-F(Q_L(t)) \Vert
    +\Vert F(Q_L(t)) -Q(t+1)\Vert.
  \end{multline*}
  Thus for any $\delta>0$,
    \begin{multline}
      \label{eq:3t}
    \PP\big(\Vert Q_L(t+1)-Q(t+1)\Vert>\delta\big)
    \leq 
    \PP\big(\Vert Q_L(t+1)-\EE[Q_L(t+1)\vert Q_L(t)] \Vert \geq\delta/3\big)\\
    +\PP(\Vert \EE[Q_L(t+1)\vert Q_L(t)]-F(Q_L(t)) \Vert \geq\delta/3\big)
    +\PP\big(\Vert F(Q_L(t)) -Q(t+1)\Vert \geq\delta/3\big).
  \end{multline}
  We bound separately the three terms on the right hand side of \eqref{eq:3t} with bounds smaller that $C/\sqrt{L}$, for a suitable constant $C$.

  The last one can be bounded by observing that by Lipschitz condition on $F$ we have
  \begin{equation*}
    \Vert F(Q_L(t)) -Q(t+1)\Vert
    =\Vert F(Q_L(t)) -F(Q(t))\Vert
    \leq \mathrm{Lip}(F) \Vert Q_L(t) -Q(t)\Vert,
  \end{equation*}
  and using inductive hypothesis \eqref{eq:qt}.

  To bound the second term of \eqref{eq:3t} we observe that the \emph{empirical process} $(Q_L(t))_{t\geq0}$ is a Markov chain with values in $\mathcal{P}(\integer_+)$.
  Its evolution can be described in the following way.
  Assume that $Q_L(t)=q$ and define for $k\in\natural$ the discrete intervals
  \begin{equation}
    \label{eq:la}
    \Lambda^q_k
    =
    \begin{cases}
      [1,L\big(q(\{0\})+q(\{1\})\big)]\cap\natural & \text{if $k=1$} \\
      \displaystyle  \Big[L\sum_{h=0}^{k-1}q(\{h\})+1,L\sum_{h=0}^{k}q(\{h\})\Big]\cap\natural & \text{if $k\geq2$},
    \end{cases}
  \end{equation}
  where we define $[a,b]=\emptyset$ when $a>b$.
  Then
  \begin{equation}
    \label{eq:ql}
    Q_L(t+1)(\{n\})
    =\frac{1}{L}\sum_{k=1}^{n+1}\sum_{x\in\Lambda^q_k}\ind(B^{L,q}_x=n+1-k).
  \end{equation}
  By equation \eqref{eq:qn} and the definition of $\Lambda_k^q$ \eqref{eq:la} we have that
  \begin{equation*}
    F(q)(\{n\})
    =\sum_{k=1}^{n+1}\frac{|\Lambda_k^q|}{L}\PP(B^q=n+1-k).
  \end{equation*}
  Thus using \eqref{eq:ql}, conditionally to $Q_L(t)=q$,
  \begin{multline*}
    \Vert \EE[Q_L(t+1)\vert Q_L(t)]-F(Q_L(t))\Vert\leq\\
    \frac{1}{2L}\sum_{n=0}^{+\infty}\sum_{k=1}^{n+1}\sum_{x\in\Lambda_k^q}\vert\PP(B^{L,q}_x=n+1-k)-\PP(B^q=n+1-k)\vert\\
        =\frac{1}{2L}\sum_{k=1}^{+\infty}\sum_{x\in\Lambda_k^q}\sum_{k=n+1}^{+\infty}\vert\PP(B^{L,q}_x=n+1-k)-\PP(B^q=n+1-k)\vert\\
    =\frac{1}{L}\sum_{k=1}^{+\infty}\sum_{x\in\Lambda_k^q}\Vert\PP(B^{L,q}_x\in\cdot)-\PP(B^q\in\cdot)\Vert
    =\Vert\PP(B^{L,q}_1\in\cdot)-\PP(B^q\in\cdot)\Vert\\
    \leq \sup_{q\in\mathcal{P}(\integer_+)}\Vert\PP(B^{L,q}_1\in\cdot)-\PP(B^q\in\cdot)\Vert.
  \end{multline*}
  Due to Condition~\ref{con:1} the last line is bounded by $C/L$.
  Markov inequality gives 
  \begin{equation*}
    \PP\big(\Vert \EE[Q_L(t+1)\vert Q_L(t)]-F(Q_L(t))\Vert>\delta/3\big)
    \leq\frac{3C}{\delta L}.
  \end{equation*}

    We now bound the first term of \eqref{eq:3t}.
    Define for any $n\in\integer_+$
  \begin{equation}
    \label{eq:M}
    M_L(n)
    :=Q_L(t+1)(\{n\})-\EE[Q_L(t+1) (\{n\})|Q_L(t)].
  \end{equation}
  Then for any $\bar n\in\natural$
  \begin{equation}
    \label{eq:2t}
    \begin{split}
           \PP\big(\Vert Q_L(t+1)&-\EE[Q_L(t+1)|Q_L(t)]\Vert>\delta/3\big)\\
    &\leq \PP\Big(\sum_{n\leq\bar n} \vert M_L(n)\vert>\delta/3\Big)+\PP\Big(\sum_{n>\bar n}\vert M_L(n)\vert>\delta/3\Big).
    \end{split}
  \end{equation}
  For the second term in \eqref{eq:2t} we observe that
  \begin{equation*}
    \vert M_L(n)\vert
    \leq Q_L(t+1)(\{n\})+\EE[Q_L(t+1)(\{n\})|Q_L(t)],
  \end{equation*}
  thus by Markov inequality
    \begin{multline*}
    \PP\Big(\sum_{n>\bar n}\vert M_L(n)\vert>\delta/3\Big)
    \leq \frac{6}{\delta} \sum_{n>\bar n} \EE[Q_L(t+1)(\{n\})]
    \leq  \frac{6}{\delta}\EE\Big[\sum_{n>\bar n} \frac{n}{\bar n} Q_L(t+1)(\{n\})\Big]\\
        \leq  \frac{6}{\delta\bar n}\EE\Big[\sum_{n=0}^{+\infty} n\, Q_L(t+1)(\{n\})\Big].
  \end{multline*}
  We observe that
  \begin{equation*}
    \sum_{n=0}^{+\infty} n\, Q_L(t+1)(\{n\})
    =\frac{1}{L}\sum_{x=1}^L\eta_x^L(t+1)
   = \frac{1}{L}\sum_{x=1}^L\eta_x^L(0),
  \end{equation*}
  because, by \eqref{eq:co}, the number of particles of the system is preserved.
  Thus by the symmetry of the distribution of $\eta^L(0)$ and the assumed uniform bound on $\EE[\eta^L_1(0)]$ we have
  \begin{equation}
    \label{eq:bn}
    \PP\Big(\sum_{n>\bar n}\vert M_L(n)\vert>\delta/3\Big)
    \leq \frac{C}{\bar n}\EE[\eta^L_1(0)]
    \leq \frac{C'}{\bar n},
  \end{equation}
  for a suitable constant $C'$.
  For the first term in \eqref{eq:2t}, using Markov and Cauchy-Schwarz inequalities we have
  \begin{equation}
    \label{eq:Ma}
    \PP\Big(\sum_{n\leq\bar n} \vert M_L(n)\vert>\delta/3\Big)
    \leq\frac{9\bar n}{\delta^2}\sum_{n\leq\bar n}\EE(M_L(n)^2).
  \end{equation}
  By definition of $M_L$ in \eqref{eq:M} we have that
  \begin{multline}
    \label{eq:va1}
    \EE(M_L(n)^2)
    =\EE\big[\EE[(Q_L(t+1)(\{n\})-\EE[Q_L(t+1)(\{n\})\vert Q^L(t)])^2\vert Q_L(t)]\big]\\
    =\EE\big[\Var[Q_L(t+1)(\{n\})\vert Q_L(t)]\big].
  \end{multline}
  Thus, using \eqref{eq:ql} conditionally to $Q_L(t)=q$
  \begin{equation*}
    Q_L(t+1)
    =\frac{1}{L}\sum_{(x,k)} Z^{L,q}_{x k},
  \end{equation*}
  where in the last sum $(x,k)\in([1,L]\times[1,n+1])\cap\natural^2$ and
  \begin{equation*}
    Z_{x k}^{L,q}:=\ind(x\in\Lambda_k^q) \ind(B^{L,q}_x=n+1-k).
  \end{equation*}
  Which implies
  \begin{equation}
    \label{eq:va}
    \Var[Q_L(t+1)(\{n\})\vert Q_L(t)=q]
    =\frac{1}{L^2}\sum_{(x,k)}\Var(Z_{x k}^{L,q})+\frac{1}{L^2}\sum_{(x,k)\neq(y,h)}\Cov(Z_{x k}^{L,q},Z_{y h}^{L,q}).
  \end{equation}
  In the sequel of this proof to keep notation simple we write
  \begin{equation*}
              \nu_L^q(n,m):=\PP(B_1^{L,q}=n,B_2^{L,q}=m),
              \qquad
              \nu^q_L(n):=\PP(B_1^{L,q}=n),
              \qquad
              \mu^q(n):=\PP(B^q=n).
  \end{equation*}
   The variance term in \eqref{eq:va} can be bounded, using the symmetry of the distribution of $B^{L,q}$, by
     \begin{equation}
    \label{eq:b1}
    \Var(Z_{x k}^{L,q})
    =\ind(x\in\Lambda_k^q)\Var(\ind(B^{L,q}_x=n+1-k))
    \leq \ind(x\in\Lambda_k^q)\nu_L^q(n+1-k).
  \end{equation}
  The covariance term in \eqref{eq:va} can be bounded, using the symmetry of the distribution of $B^{L,q}$ by
   \begin{multline}
    \label{eq:b2}
    \Cov(Z_{x k}^{L,q},Z_{y h}^{L,q})\\
    \leq \ind(x\in\Lambda^q_k) \ind(y\in\Lambda^q_h)\vert\nu_L^q(n+1-k,n+1-h)-\nu_L^q(n+1-k) \nu_L^q(n+1-h)\vert.
  \end{multline}
  Using the bounds \eqref{eq:b1} and \eqref{eq:b2} in \eqref{eq:va}, summing on $x$ and $y$ and changing the variables $n+1-k\mapsto k$ and   $n+1-h\mapsto h$ we arrive to
    \begin{multline*}
      \Var[Q_L(t+1)(\{n\})\vert Q_L(t)=q]\\
    \leq\frac{1}{L}\sum_{k=0}^{n}\frac{\vert\Lambda_{n+1-k}^q\vert}{L} \nu^q_L(k)
    +\sum_{k,h=0}^{n}\frac{\vert\Lambda_{n+1-k}^q\vert \vert\Lambda_{n+1-h}^q\vert}{L^2}\vert \nu^q_L(k,h)-\nu^q_L(k) \nu^q_L(h)\vert.
  \end{multline*}
  So, exchanging the sums to obtain the second inequality below, we get
  \begin{multline}
    \label{eq:bva}
    \sum_{n\leq\bar n}\Var[Q_L(t+1)(\{n\})\vert Q_L(t)=q]
    \leq\sum_{n=0}^{+\infty}\Var[Q_L(t+1)(\{n\})\vert Q_L(t)=q]\\
       \leq\frac{1}{L}\sum_{k=0}^{+\infty}\nu^q_L(k)\sum_{n=k}^{+\infty}\frac{\vert\Lambda_{n+1-k}^q\vert}{L}
    +\sum_{k,h=0}^{+\infty}\vert\nu_L^q(k,h)-\nu^q_L(k) \nu^q_L(h)\vert\sum_{n=k\vee h}^{+\infty}\frac{\vert\Lambda_{n+1-k}^q\vert \vert\Lambda_{n+1-h}^q\vert}{L^2}.
  \end{multline}
  Since $\{\Lambda_k^q\colon k\in\natural\}$ is a partition of $[1,L]\cap\natural$:
  \begin{equation*}
    \sum_{n=k}^{+\infty}\frac{\vert\Lambda_{n+1-k}^q\vert}{L}=1
  \end{equation*}
  and
  \begin{equation*}
    \sum_{n=k\vee h}^{+\infty}\frac{\vert\Lambda_{n+1-k}^q\vert \vert\Lambda_{n+1-h}^q\vert}{L^2}
    \leq\sum_{n=k\vee h}^{+\infty}\frac{\vert\Lambda_{n+1-k\vee h}^q\vert}{L}
    =1,
  \end{equation*}
  thus by \eqref{eq:bva} and Condition~\ref{con:1} there exists a constant $C$ such that
  \begin{multline*}
    \sum_{n\leq\bar n}\Var[Q_L(t+1)(\{n\})\vert Q_L(t)=q]\\
    \leq \frac{1}{L}+\sum_{k,h=0}^{+\infty}\vert \nu^q_L(k,h)-\nu^q_L(k) \nu^q_L(h)\big\vert
    \leq  \frac{1}{L}+6\Vert\nu_L^q-\mu^q\otimes\mu^q\Vert
    \leq \frac{C}{L}.
  \end{multline*}
  By \eqref{eq:Ma} and \eqref{eq:va1} there is a constant $C$
  \begin{equation*}
    \PP\Big(\sum_{n\leq\bar n} \vert M_L(n)\vert>\delta/3\Big)
    \leq \frac{C\bar n}{L}.
  \end{equation*}
   Plugging this bound and the bound il \eqref{eq:bn} in \eqref{eq:2t} and optimizing on $\bar n$ we arrive to
   \begin{equation*}
     \PP\Big(\Vert M_L(n)\Vert>\delta/3\Big)
     \leq C\Big(\frac{1}{\bar n}+\frac{\bar n}{L}\Big)
   \leq \frac{2C+1}{\sqrt{L}},
   \end{equation*}
  for some constant $C$.
  
\end{proof}

\begin{re}
  It is natural to wonder if a better decay rate, in \eqref{eq:co1} and in the initial condition \eqref{eq:ma1}, may improve this rate.
  This can be true in some cases, however simply replacing \eqref{eq:co1} and \eqref{eq:ma1} with stronger conditions our proof does not provide automatically a better rate in \eqref{eq:ma2}.
\end{re}

\section{Classical occupancy models}
\label{sec:app}

In this section we consider the GRBB process when $B^{L,q}$ is distributed according to the Fermi-Dirac, Maxwell-Boltzmann or Bose-Einstein statistics.
The first two models are examples of rather natural strategies to reassign customers in a Jackson network.
In the first case the customers are efficiently reassigned in different queues. 
In the second case every customer is independently reassigned to a queue.
The last case is included as it is a simple example of a nontrivial dependent random reassigned rule.

We here show that the hypothesis of Theorem~\ref{teo:mai} holds for these three classical occupancy models.
In particular we prove that Condition~\ref{con:1} is satisfied for these models when $\mu^q$ is the natural limit point of the one site marginal of $B^{L,q}$.
The obtained bounds will have the decay of the right hand side of Condition~\ref{con:1} when the number of particles of the considered system is proportional to $L$.

\subsection{Fermi-Dirac statistics}
\label{sec:FD}

We say that the random vector $X=(X_1,\dots,X_L)$ follows the Fermi-Dirac statistics with $L$ sites and  $N\leq L$ particles if 
\begin{equation*}
  \PP(X_1=x_1,\dots,X_L=x_L)
  =
  \begin{cases}
     \frac{1}{\binom{L}{N}} & \text{if $(x_1,\dots,x_L)\in\{0,1\}^L$ and $\sum_{k=1}^Lx_k=N$}\\
      0 & \text{otherwise.}
  \end{cases}
\end{equation*}
Given $q\in\mathcal{P}(\integer_+)$, let $\mu^q\in\mathcal{P}(\integer_+)$ be the Bernoulli distribution with parameter $1-q(\{0\})$ and assume that $B^{L,q}$ follows the Fermi-Dirac statistics with $L$ sites and  $(1-q(\{0\}))L$ particles.
The map $\psi(q):=\mu^q$ is, in this case, 1-Lipschitz: given $q,q'\in\mathcal{P}(\integer_+)$
\begin{equation*}
  \begin{split}
      \Vert \mu^q-\mu^{q'}\Vert
  =\frac{1}{2}\sum_{s\in\{0,1\}}\vert \mu^q(\{s\})- \mu^{q'}(\{s\})\vert
  =\vert q(\{0\})- q'(\{0\})\vert\\
  =\frac{1}{2}\vert q(\{0\})- q'(\{0\})\vert+\frac{1}{2}\Big\vert \sum_{s\in\natural}q(\{s\})- \sum_{s\in\natural}q'(\{s\})\Big\vert
  \leq \Vert q-q'\Vert.
  \end{split}
\end{equation*}
To prove that Condition~\ref{con:1} holds we need the following result.
\begin{teo}
  \label{teo:FD}
  Assume that $X$ follows the Fermi-Dirac statistics with $L$ sites and $N$ particles.
  Denote with $\gamma_L^N\in\mathcal{P}(\integer_+^2)$ the distribution of $(X_1,X_2)$ and let $\lambda^{N/L}\in\mathcal{P}(\integer_+)$ be the Bernoulli distribution with parameter $N/L$.
  Then,
  \begin{equation*}
    \Vert\gamma_L^N-\lambda^{N/L}\otimes\lambda^{N/L}\Vert
    =\frac{2N}{L(L-1)}\Big(1-\frac{N}{L}\Big)
  \end{equation*}
\end{teo}

\begin{proof}
  We have that 
  \begin{equation}
    \label{eq:FD1}
    \Vert\gamma_L^N-\lambda^{N/L}\otimes\lambda^{N/L}\Vert
    =\frac{1}{2}\sum_{h,k=0}^1\vert\PP(X_1=h,X_2=k)-\PP(X_1=h)\PP(X_2=k)\vert.
  \end{equation}
  Observing that
  \begin{equation*}
    \PP(X_1=h,X_2=k)-\PP(X_1=h)\PP(X_2=k)
    =(-1)^{h+k}\Cov(X_1,X_2),
    \qquad
    h,k\in\{0,1\},
  \end{equation*}
  and, because $X_1+\dots+X_L=N$,
  \begin{equation*}
    \Cov(X_1,X_2)
    =-\frac{1}{L-1}\Var(X_1)
    =-\frac{N}{L(L-1)}\Big(1-\frac{N}{L}\Big),
  \end{equation*}
  the result follows.
\end{proof}
Theorem~\ref{teo:FD}  assures that Condition~\ref{con:1} holds as 
\begin{equation*}
  \sup_{q\in\mathcal{P}(\integer_+)}\Vert\nu_L^q-\mu^q\otimes\mu^q\Vert
  = \sup_{N/L\leq 1}\Vert\gamma_L^N-\lambda^{N/L}\otimes\lambda^{N/L}\Vert
  \leq \frac{1}{L}.
\end{equation*}

  We observe that in the present case the GRBB process $(\eta^L(t))_{t\geq0}$ started with $N\leq L$ particles is ergodic and reversible with invariant measure given by the Fermi-Dirac statistics with $L$ sites and $N$ particles.
  Thus in this case propagation of chaos holds also at equilibrium. 
  If $N>L$ the GRBB process is not irreducible as there are blocked configurations.
  For completeness below we give an upper bound for the mixing time for the GRBB process in this case.

\begin{pro}
  Let $(\eta^L(t))_{t\geq0}$ be the GRBB process when $B^{L,q}$ follows the Fermi-Dirac statistics with $L$ sites and  $(1-q(\{0\}))L$ particles.
  Assume that $\eta^L(0)\in\integer_+^L$ and
  \begin{equation*}
    N:=
    \sum_{x=1}^L\eta^L_x(0)\in[2, L/2].
  \end{equation*}
  Then $(\eta^L(t))_{t\geq0}$ is ergodic and its invariant measure $\pi^{N,L}$ is the Fermi-Dirac distribution with $L$ sites and $N$ particles.
  Furthermore
  \begin{equation*}
    t_{\mathrm{mix}}
    :=\inf\Big\{t\geq0\colon \Vert\PP(\eta^L(t)\in\cdot)-\pi^{N,L}\Vert\leq1/4\Big\}
    \leq 13 -4L\log\Big(1-\frac{N-1}{L}\Big).
  \end{equation*}
  where $\PP(\eta^L(t)\in\cdot)$ is the distribution of $\eta^L(t)$.  
\end{pro}

\begin{proof}
  Define the decreasing sequence of events $A_1\supseteq A_2\supseteq\dots\supseteq A_N$, where
  \begin{equation*}
    A_n
    :=\Big\{\eta\in\integer_+^L\colon \sum_{x=1}^L\ind(\eta_x>0)
    \geq n\Big\},
  \end{equation*}
  and the increasing sequence $0:=\tau_1\leq\dots\leq\tau_N$ of hitting times
     \begin{equation*}
      \tau_n
      :=\inf\big\{t\geq\tau_{n-1}\colon \eta^L(t)\in A_n \}.
    \end{equation*}
    We notice that $\tau_N$ is the first time such that in every site there is at most one particle and that at time $t=\tau_N+1$ the system is distributed with its stationary measure independently from its state at time $t=\tau_N$.
    So $\tau_N+1$ is a \emph{strong stationary time} and we can use Proposition~6.11 of \cite{Le:Pe} to get
      $t_{\mathrm{mix}}\leq\inf\big\{t\geq0\colon \PP(\tau_N+1>t)\leq1/4\big\}$.
      By Markov inequality we have $t_{\mathrm{mix}}\leq4\EE(\tau_N+1)+1$.
    To bound $\EE(\tau_N)$ we introduce the random variables $\sigma_n:=\tau_{n+1}-\tau_{n}$ so that
    \begin{equation}
      \label{eq:sig}
      \EE(\tau_N)=\sum_{n=1}^{N-1}\EE(\sigma_n).
    \end{equation}
    Now observe that $\sigma_n>t$ if and only if $\eta^L(\tau_n+t)\in A_n\setminus A_{n+1}$.
    In fact if $\tau_{n+1}>\tau_n +t$ then $\eta^L(\tau_n +t)\not\in A_{n+1}$ but $\eta^L(\tau_n+t)\in A_{n}$ because $\tau_n+t>\tau_n$.
    Thus $\eta^L(\tau_n+t)\in A_{n}\setminus A_{n+1}$.
    On the other hand if $\eta^L(\tau_n+t)\in A_{n}\setminus A_{n+1}$ then $\eta^L(\tau_n+t)\not\in A_{n+1}$ thus $\tau_{n+1}>\tau_n +t$.
    
     By strong Markov property
     \begin{equation}
       \label{eq:sm}
       \begin{split}
         \PP(\sigma_n>t)&=\PP(\eta^L(\tau_n+t)\in A_n\setminus A_{n+1})\\
         &=\sum_{\eta\in A_n\setminus A_{n+1}}\PP(\eta^L(1)\in A_n\setminus A_{n+1}|\eta^L(0)=\eta)\PP(\eta^L(\tau_n+t-1)=\eta).
       \end{split}
     \end{equation}
          A configuration in $A_n\setminus A_{n+1}$ is a configuration where there are $n$ \emph{mobile} particles, $N-n$ \emph{blocked} particles and $L-n$ empty sites.    
     So, if $\eta^L(0)=\eta\in A_n\setminus A_{n+1}$, we have that $\eta^L(1)\in A_n\setminus A_{n+1}$ if and only if the process puts a mobile particle on every site occupied by a blocked particle.
     Let $p_{k,m}$ be the probability that, following the Fermi-Dirac statistics with $L$ sites and $m\leq L$ particles, the sites $\{1,\dots,k\}$ are occupied.
      Thus if $\bar k\in\{1,\dots,N-n\}$ is the number of sites occupied by the blocked particles in the configuration $\eta$, then
      \begin{equation*}
        \PP(\eta^L(1)\in A_n\setminus A_{n+1}|\eta^L(0)=\eta)=p_{\bar k,n}\leq p_{1,n}=n/L.
      \end{equation*}
      Plugging this bound into \eqref{eq:sm} we get
     \begin{align*}
       \PP(\sigma_n>t)
       \leq\frac{n}{L} \sum_{\eta\in A_n\setminus A_{n+1}}\PP(\eta^L(\tau_n+t-1)=\eta)
       &=\frac{n}{L}\PP(\eta^L(\tau_n+t-1)\in A_n\setminus A_{n+1})\\
       &=\frac{n}{L}\PP(\sigma_n>t-1),
     \end{align*}
     which, iterating, implies
     \begin{equation*}
       \PP(\sigma_n>t)
       \leq\left(\frac{n}{L}\right)^t.
     \end{equation*}
     So, by summing the geometric series, $\EE(\sigma_n)\leq L/(L-n)$ and by \eqref{eq:sig}
     \begin{equation*}
       \EE(\tau_N)\leq
       L\sum_{n=L+1-N}^{L-1}\frac{1}{n}\leq
       \frac{L}{L-N+1}+L\log\left(\frac{L-1}{L-N+1}\right).
     \end{equation*}
     Thus
     \begin{equation*}
       t_{\mathrm{mix}}
       \leq4\EE(\tau_N+1)+1
       \leq5+\frac{4 L}{L-N+1}-4L\log\Big(1-\frac{N-1}{L}\Big)
     \end{equation*}
     and the result follows. 
\end{proof}

\subsection{Maxwell-Boltzmann statistics}
\label{sec:MB}

We say that the random vector $X=(X_1,\dots,X_L)$ follows the Maxwell-Boltzmann statistics with $L$ sites and particles $N$ if 
\begin{equation*}
  \PP(X_1=x_1,\dots,X_L=x_L)
  =
  \begin{cases}
   \binom{N}{x_1,\dots,x_L}\frac{1}{L^N}, &\text{if $(x_1,\dots,x_L)\in\integer_+^L$ and $\sum_{k=1}^Lx_k=N$,}\\
   0 &\text{otherwise.}
  \end{cases}
\end{equation*}
Given $q\in\mathcal{P}(\integer_+)$, let $\mu^q\in\mathcal{P}(\integer_+)$ be the Poisson distribution with parameter $1-q(\{0\})$  and assume that $B^{L,q}$ follows the Maxwell-Boltzmann statistics with $L$ sites and  $(1-q(\{0\}))L$ particles.
In this case the GRBB process is the RBB process studied in \cite{Be:Cl:Na:Pa:Po} and \cite{Ca:Po}.
To apply Theorem~\ref{teo:mai} we have to show that Condition~\ref{con:1} holds and the map $\psi$ is Lipschitz.
To check Lipschitz property take $q,q'\in\mathcal{P}(\integer_+)$ and set $q_0:=q(\{0\})$, $q'_0:=q'(\{0\})$.
Then
\begin{equation*}
  \Vert\mu^q-\mu^{q'}\Vert
  =\frac{1}{2}\sum_{k=0}^{+\infty}\frac{1}{k!}\big|(1-q_0)^ke^{-(1-q_0)}-(1-q'_0)^ke^{-(1-q'_0)}\big|
  \leq \frac{1}{2}\sum_{k=0}^{+\infty}\frac{k+1}{k!}|q_0-q'_0|
  \leq 3 \Vert q-q'\Vert.
\end{equation*}
Condition~\ref{con:1} is implied by the following result.

\begin{teo}
  Assume that $X$ follows the Maxwell-Boltzmann statistics with $L$ sites and $N$ particles.
  Denote with $\gamma_L^N\in\mathcal{P}(\integer_+^2)$ the distribution of $(X_1,X_2)$ and let $\lambda^{N/L}\in\mathcal{P}(\integer_+)$ be the Poisson distribution with parameter $N/L$.
  Then,
  \begin{equation*}
    \Vert\gamma_L^N-\lambda^{N/L}\otimes\lambda^{N/L}\Vert
    \leq\frac{4N}{L^2}.
  \end{equation*}
\end{teo}

\begin{proof}
  Let $\bar\gamma_L^N:=\PP(X_1\in\cdot)$ be the one site marginal of $\gamma_L^N$.
    \begin{multline}
      \label{eq:gam}
    \Vert\gamma_L^N-\lambda^{N/L}\otimes\lambda^{N/L}\Vert
   \leq \Vert\gamma_L^N-\bar\gamma_L^N\otimes \bar\gamma_L^N\Vert
   +\Vert \bar\gamma_L^N\otimes \bar\gamma_L^N-\bar\gamma_L^N\otimes\lambda^{N/L}\Vert
   +\Vert \bar\gamma_L^N\otimes\lambda^{N/L}-\lambda^{N/L}\otimes\lambda^{N/L}\Vert\\
   = \Vert \gamma_L^N-\bar\gamma_L^N\otimes \bar\gamma_L^N\Vert
   +2 \Vert \bar\gamma_L^N-\lambda^{N/L}\Vert,
  \end{multline}
  where, in the last line, we used the fact that for arbitrary probability measures $\gamma$, $\mu$ and $\nu$:
  \begin{equation*}
    \Vert\gamma\otimes\mu-\gamma\otimes\nu\Vert
    = \Vert\mu-\nu\Vert.
  \end{equation*}
  We bound separately the 2 terms in equation \eqref{eq:gam}.

  The second term is bounded using Poisson approximation of binomial distribution (see for example \cite{Sh} \S 12) so that
    \begin{equation}
      \label{eq:MB2}
      \Vert \bar\gamma_L^N-\lambda^{N/L}\Vert
      \leq\frac{N}{L^2}.
    \end{equation}
    To bound the first term in \eqref{eq:gam} we construct a coupling, namely we define $(X_1,X_2)\sim\gamma_L^N$ and $(Y_1,Y_2)\sim\bar\gamma_L^N\otimes\bar\gamma_L^N$.
    We take $X_1$ as a binomial random variable with parameters $N$ and $1/L$, \ie $X_1\sim\bar\gamma^N_L$ and consider $U_1,\dots U_N$ \iid uniformly distributed random variables with values in $[0,1]$ \emph{independent from $X_1$}.
    Define next
       \begin{equation*}
      X_2
      :=\sum_{k=1}^{N-X_1}\ind\Big(U_k\leq\frac{1}{L-1}\Big),
     \end{equation*}
     where here and in the sequel we use the convention that $\sum_{k=1}^0:=0$.
    An elementary computation shows that $(X_1,X_2)$ has the distribution of the two components of a random vector following the Maxwell-Boltzmann statistics with $N$ particles an $L$ sites, \ie $(X_1,X_2)\sim\gamma_L^N$.
    Define $Y_1:=X_1$ and 
    \begin{equation*}
      Y_2
      :=\sum_{k=1}^{N}\ind\Big(U_k\leq\frac{1}{L}\Big).
    \end{equation*}
    Clearly $Y_1$ and $Y_2$ are \iid random variables with common binomial distribution with parameters $N$ and $1/L$, \ie $(Y_1,Y_2)\sim\bar\gamma_L^N\otimes\bar\gamma_L^N$.
    Thus, defining the event
    \begin{equation*}
      A
      :=
      \Big\{\sum_{k=1}^{N-X_1}\ind\Big(U_k\leq\frac{1}{L-1}\Big)= \sum_{k=1}^{N-X_1}\ind\Big(U_k\leq\frac{1}{L}\Big)\Big\},
    \end{equation*}
    we have,
    \begin{multline}
      \label{eq:MB3}
      \Vert\gamma_L^N-\bar\gamma_L^N\otimes\bar\gamma_L^N\Vert
      \leq\PP((X_1,X_2)\neq(Y_1,Y_2))
      =\PP(X_2\neq Y_2)\\
      =\PP(X_2\neq Y_2,A)+\PP(X_2\neq Y_2,A^c).
    \end{multline}
    We bound separately the two terms on the right hand side of equation \eqref{eq:MB3}.
    For the first one, we observe that the only way to have that $A$ occurs and $X_2\neq Y_2$ is that
    \begin{equation*}
      \sum_{k=N-X_1+1}^N\ind\Big(U_k\leq \frac{1}{L}\Big)>0.
    \end{equation*}
    Thus, by the independence of $X_1,U_1,\dots,U_N$:
    \begin{equation*}
      \PP(X_2\neq Y_2,A\Big)
      =1-\PP\Big(U_{N-X_1+1}>\frac{1}{L},\dots,U_N >\frac{1}{L}\Big)
      =1-\sum_{n=0}^N\PP(X_1=n)\Big(1-\frac{1}{L}\Big)^n.
    \end{equation*}
    Using the binomial distribution of $X_1$, an explicit computation shows that
    \begin{equation}
      \label{eq:MB4}
      \PP(X_2\neq Y_2,A)
      =1-\Big(1-\frac{1}{L^2}\Big)^N
      \leq \frac{N}{L^2}.
    \end{equation}
    For the second term in equation \eqref{eq:MB3} using again the independence of $X_1,U_1,\dots,U_N$ we can write
    \begin{multline}
      \label{eq:MB5}
      \PP(X_2\neq Y_2,A^c)\\
      \leq\PP(A^c)
      =\sum_{n=0}^N\PP(X_1=n)\PP\Big(\frac{1}{L}< U_k\leq \frac{1}{L-1} \text{ for some } k\in\{1,\dots,N-n\}\Big)\\
      \leq \sum_{n=0}^N\PP(X_1=n)\sum_{k=1}^{N-n}\PP\Big(\frac{1}{L}< U_k\leq \frac{1}{L-1}\Big)
      =\frac{N-\EE(X_1)}{L(L-1)}
      =\frac{N}{L^2}.
    \end{multline}
    Plugging bounds \eqref{eq:MB5} and \eqref{eq:MB4} into equation \eqref{eq:MB3} we obtain
    \begin{equation*}
      \Vert\gamma_L^N-\bar\gamma_L^N\otimes\bar\gamma_L^N\Vert
      \leq \frac{2N}{L^2},
    \end{equation*}
    which together with the bound \eqref{eq:MB2} and equation \eqref{eq:gam} proves the result.
\end{proof}

\begin{re}
  We studied  the mixing time for the GRBB process when $B^{L,q}$ follows the Maxwell-Boltzmann statistics with $L$ sites and  $(1-q(\{0\}))L$ particles in \cite{Ca:Po:1}.
  We proved that the process with $L$ sites and $r L$ particles has mixing time is of order $L$.
\end{re}

\subsection{Bose-Einstein statistics}
\label{sec:BE}

We say that the random vector $X=(X_1,\dots,X_L)$ follows the Bose-Einstein statistics with $L\in\natural$ sites and particles $N\in\natural$ if 
\begin{equation*}
  \PP(X_1=x_1,\dots,X_L=x_L)
  =
  \begin{cases}
    \frac{1}{\binom{L+N-1}{N}}, &\text{if $(x_1,\dots,x_L)\in\integer_+^L$ and $\sum_{k=1}^Lx_k=N$,}\\
   0 &\text{otherwise.}
  \end{cases}
\end{equation*}
We note that a sample of this distribution can be obtained drawing $N$ balls from an \emph{$L$-color Pólya urn} with \emph{double replacement}. 
Namely, put $L$ numbered balls in an empty urn.
Then one ball is randomly extracted from the urn, its number recorded and it is returned into the urn with an additional ball with the same number.
Repeating this procedure $N$ times yields a random vector $X=(X_1,\dots,X_L)$, where $X_k$ is the number of times that a ball with number $k$ has been drawn in the $N$ extractions.
Then $X$ follows the Bose-Einstein statistics with $L$ sites and $N$ particles, see \eg Lemma~2.7 of \cite{Le:Pe}.
This connection of the Bose-Einstein statistics with the $L$-color Pólya urn will be a key ingredient in the proof of Theorem~\ref{teo:BE} below.

Given $q\in\mathcal{P}(\integer_+)$, let $\mu^q\in\mathcal{P}(\integer_+)$ be the geometric distribution supported on $\integer_+$ and parameter $1/(2-q(\{0\}))$ and assume that $B^{L,q}$ follows the Bose-Einstein statistics with $L$ sites and  $(1-q(\{0\}))L$ particles.
To apply Theorem~\ref{teo:mai} we have to show that Condition~\ref{con:1} holds and the map $\psi$ is Lipschitz.
To check Lipschitz property proceeding as in Section~\ref{sec:MB} we get for any $q,q'\in\mathcal{P}(\integer_+)$ 
\begin{equation*}
  \Vert\mu^q-\mu^{q'}\Vert
  \leq 4 \Vert q-q'\Vert.
\end{equation*}
Condition~\ref{con:1} is implied by the following result.

\begin{teo}
  \label{teo:BE}
  Assume that $X$ follows the Bose-Einstein statistics with $L$ sites and $N$ particles.
  Denote with $\gamma_L^N\in\mathcal{P}(\integer_+^2)$ the distribution of $(X_1,X_2)$ and let $\lambda^{N/L}\in\mathcal{P}(\integer_+)$ be the geometric distribution with support $\integer_+$ and parameter $1/(1+N/L)$:
  \begin{equation*}
    \lambda^{N/L}(\{k\})
    :=\ind(k\in\integer_+)\frac{1}{1+N/L}\Big(1-\frac{1}{1+N/L}\Big)^k.
  \end{equation*}
  Then
  \begin{equation*}
    \Vert\gamma_L^N-\lambda^{N/L}\otimes\lambda^{N/L}\Vert
    \leq\frac{14 N}{L^2}.
  \end{equation*}
\end{teo}

\begin{proof}
  Let $\bar\gamma_L^N:=\PP(X_1\in\cdot)$ be the one site marginal of $\gamma_L^N$.
  As in \eqref{eq:gam} we get
    \begin{equation}
      \label{eq:gam1}
    \Vert\gamma_L^N-\lambda^{N/L}\otimes\lambda^{N/L}\Vert
   \leq \Vert \gamma_L^N-\bar\gamma_L^N\otimes \bar\gamma_L^N\Vert
   +2 \Vert \bar\gamma_L^N-\lambda^{N/L}\Vert.
  \end{equation}
  We bound separately the two terms in the right hand side of inequality \eqref{eq:gam1}.
  For the second one we use Theorem~3 of \cite{Br:Ph} to get
  \begin{equation}
    \label{eq:BE1}
    \Vert \bar\gamma_L^N-\lambda^{N/L}\Vert
    \leq \frac{6}{L}.
  \end{equation}
  To bound the first term in \eqref{eq:gam1} we construct a coupling using the $L$-color Pólya urn with double replacement.
  We define $(X_1,X_2)\sim\gamma_L^N$ and $(Y_1,Y_2)\sim\bar\gamma_L^N\otimes\bar\gamma_L^N$.
    We generate $X_1$ as the 1 site marginal of $X$, namely $X_1\sim\bar\gamma^N_L$ and define $Y_1:=X_1$.

    Given $X_1=n$, to generate $X_2$ and $Y_2$ we consider two urns named urn A and urn B.
    Initially in urn A there are $L-1$ balls numbered from $2$ to $L$, while in urn B there are $L$ balls numbered from $1$ to $L$.
    We distinguish two cases, case $n<N$ and case $n=N$.
    
    \medskip
    \emph{Case} $n<N$.
    We draw a ball from urn A, assume that it is ball $k$, then we \emph{try} to extract the same ball from urn B.
     To this end, independently, we generate a Bernoulli random variable with success probability $(L-1)/L$.
     In case of \emph{success} we extract the ball $k$ from urn B;
     in case of \emph{failure} we extract the ball 1 from urn B.
     The extracted balls are then returned in their urns with a ball with the same number.

     The next extractions are defined inductively.
     Assume that $t$ extractions have been made with $0<t<N-n$.
     For $k\in\{2,\dots,L\}$ let $t_k$ be the number of balls $k$ drawn from urn A in the $t$ extractions and $f_k$, the number of times that the attempt to extract the same ball $k$ from urn B failed in the $t$ extractions.
     For the $t+1$ extraction we draw a ball from urn A and make a \emph{test} by generating a Bernoulli random variable with success probability
    \begin{equation}
      \label{eq:BE4}
      \frac{(L+t-1)(1+t_k-f_k)}{(L+t)(1+t_k)}.
    \end{equation}
     In case of \emph{success} we extract a ball $k$ from urn B;
     in case of \emph{failure} we extract a ball 1 from urn B.
     We then use double replacement.

     We iterate the preceding rule until $t=N-n$.
     Next we go on extracting a ball from urn B with double replacement
     for other $n$ steps.
      Define $X_2$ as the number of times that ball 2 has been drawn from urn A and  $Y_2$ as the \emph{total} number of times that ball 2 has been drawn from urn B.

    \medskip
    \emph{Case} $n=N$.
     We define $X_2=0$.
     To define $Y_2$ we draw a ball from urn B with double replacement.
     Define $Y_2$ as the number of times that ball 2 has been drawn from urn B. 

    \bigskip
    We claim that in both cases $(X_1,X_2)\sim\gamma_L^N$ and $(Y_1,Y_2)\sim\bar\gamma_L^N\otimes\bar\gamma_L^N$.
    The first claim is a Bose-Einstein statistics property: for any $n,m\in\integer_+$ such that $n+m\leq N$
    \begin{equation}
      \label{eq:BE6}
      \begin{split}
           \PP(X_1=n,X_2=m)
      &=\PP(X_2=m\vert X_1=n) \PP(X_1=n)\\
      &=\bar\gamma_L^N(\{n\})\bar\gamma_{L-1}^{N-n}(\{m\})
      =\gamma_L^N(\{n\}\times\{m\}),
      \end{split}
    \end{equation}
    where $\gamma_{L-1}^{0}:=\delta_0$.

    To prove that $(Y_1,Y_2)\sim\bar\gamma_L^N\otimes\bar\gamma_L^N$ note that for any $n,m\in\{0,\dots,N\}$
    \begin{equation*}
      \PP(Y_1=n,Y_2=m)
      =\PP(Y_2=m\vert X_1=n)\PP(X_1=n),
    \end{equation*}
    so the result follows if we can show that
    \begin{equation}
      \label{eq:BE2}
      \PP(Y_2\in\cdot\vert X_1=n)
      =\bar\gamma_L^N,
      \qquad\qquad
      n\in\{0,\dots,N\}.
    \end{equation}
    To prove \eqref{eq:BE2} we will show that $Y_2$, conditionally to $X_1=n$, is the number of balls 2 extracted in $N$ draws from an $L$-color Pólya urn.
    Then, by Lemma~2.7 of \cite{Le:Pe}, after $N$ extractions the number of times that ball 2 has been extracted has $\bar\gamma_L^N$ distribution.
    Let $T_k(t)$, $k\in\{1,\dots,L\}$, be the number of balls $k$ drawn from an $L$-color Pólya urn in $t$ steps.
    Then $T(t):=(T_1(t),\dots,T_L(t))$ is a homogeneous time Markov chain, see for example \cite{Le:Pe} \S2.4, called \emph{Pólya urn process}.

    If $n=N$ \eqref{eq:BE2} holds because by construction B is a Pólya urn.

    Assume that $n<N$.
    To show that B is a Pólya urn we must verify that balls are uniformly drawn from B at each draw.

    To compute the probability to extract a ball $k$ from B we observe that
    if $k\in\{2,\dots,L\}$, $k$ is extracted from B if and only if it is extracted from A and the test is a success, while ball 1 is extracted from B if and only if the test fails.
    So at the first extraction, ball $k$, $k\in\{2,\dots,L\}$ is chosen from urn B with probability
    \begin{equation*}
       \frac{1}{L-1}\frac{L-1}{L}
      =\frac{1}{L}.
    \end{equation*}
    Ball $1$ is chosen with probability $1/L$.
    
    Assume that $t$ extractions have been made with $0<t<N-n$.
     Recall that $t_k$ and $f_k$, $k\in\{2,\dots,L\}$, denote the number of balls $k$ drawn from urn A and the number of times that the attempt to extract the same ball $k$ from urn B failed respectively in the $t$ extractions.
     Then urn A has $t_k+1$ balls $k$ while urn B has $1+f_2+\dots+f_L$ balls 1 and $1+t_k-f_k$ balls $k$, $t_2+\dots+t_L=t$.

      Thus at the $t+1$ extraction, ball $k$ with $k\in\{2,\dots,L\}$ is chosen from urn B with probability
    \begin{equation*}
       \frac{t_k+1}{L-1+t}\frac{(L+t-1)(1+t_k-f_k)}{(L+t)(1+t_k)}
      =\frac{1+t_k-f_k}{L+t}
    \end{equation*}
    while ball $1$ is chosen with probability
    \begin{equation*}
      1-\sum_{k=1}^L \frac{1+t_k-f_k}{L+t}
      =\frac{1+f_2+\dots+f_L}{L+t}.
    \end{equation*}
    In any case the balls are chosen uniformly. 
    We iterate the preceding rule until $t=N-n$.
    At this time the urn B is an $L$-color Pólya urn process after $N-n$ steps.
    Next we go on extracting a ball from urn B with double replacement for other $n$ steps.
    So, after $N$ extraction, B is an $L$-color Pólya urn process after $N$ steps and $Y_2$ is the number of times ball 2 has been extracted in $N$ draws, \ie \eqref{eq:BE2} holds.

    We observe that, as in \eqref{eq:MB3},
    \begin{equation*}
      \Vert\gamma_L^N-\bar\gamma_L^N\otimes\bar\gamma_L^N\Vert
      =\PP(X_2\neq Y_2).
    \end{equation*}
    Define the event $D$ as ``a ball 2 has been drawn from urn A in the first $N-X_1$ extractions and the associated test failed'' and the event $E$ as ``a ball 2 has been drawn from urn B in the last $X_1$ extractions''.
    Then
    \begin{equation}
      \label{eq:BE5}
      \PP(X_2\neq Y_2)
      \leq \PP(D\cup E)
      \leq\sum_{n=0}^N \PP(D|X_1=n)\PP(X_1=n)+\sum_{n=0}^N \PP(E|X_1=n)\PP(X_1=n).
    \end{equation}
    If  $n=N$ the first term in \eqref{eq:BE5} is zero.
    If $n<N$ we write
    \begin{equation}
      \label{eq:BE8}
      \PP(D|X_1=n)
      =\sum_{j=1}^{N-n}\PP(D|X_1=n,X_2=j) \PP(X_2=j|X_1=n).
    \end{equation}
    Observe that by \eqref{eq:BE6}
    \begin{equation}
      \label{eq:BE7}
      \PP(X_2=j|X_1=n)
      =\bar\gamma_{L-1}^{N-n}(\{j\}),
    \end{equation}
    and that if $X_1=n$ and $X_2=j$ the event $D^c$ occurs if and only if all the $j$ tests associated to the extractions of a ball 2 are success.
    By \eqref{eq:BE4} the probability that the test at extraction $t$ is a success, if all the preceding tests are successful (\ie $f_2=0$), is 
    \begin{equation*}
      \frac{L+t-1}{L+t}\geq
      1-\frac{1}{L}.
    \end{equation*}
    Thus
    \begin{equation}
      \PP(D|X_1=n,X_2=j)
      =1-\PP(D^c|X_1=n,X_2=j)
      \leq 1- \Big(1-\frac{1}{L}\Big)^j
      \leq \frac{j}{L}.
    \end{equation}
    By plugging this bound and equation \eqref{eq:BE7} into equation \eqref{eq:BE8}
 we get
    \begin{equation*}
      \PP(D|X_1=n)
      \leq \frac{1}{L}\sum_{j=1}^{N-n}j\bar\gamma_{L-1}^{N-n}(\{j\})
      =\frac{N-n}{L(L-1)}.
    \end{equation*}
    Thus
    \begin{equation}
      \label{eq:BE9}
      \sum_{n=0}^N \PP(D|X_1=n)\PP(X_1=n)
      \leq \frac{N}{L(L-1)}-\frac{N}{L^2(L-1)}
      =  \frac{N}{L^2}.
    \end{equation}
    
    We consider the second term in \eqref{eq:BE5}.
    If $X_1=n$, B is an $L$-color Pólya urn from which $N-n$ balls have been drawn.
    Because the $L$-color Pólya process is an homogeneous Markov chain, the number of times that a ball 2 will be draw in the last $n$ extractions has distribution $\bar\gamma_{L}^n$.
    Thus 
    \begin{equation*}
      \PP(E|X_1=n)
      =1-\bar\gamma_L^n(\{0\})
      =1-\frac{\binom{n+L-2}{n}}{\binom{n+L-1}{n}}
      =\frac{n}{n+L-1}
      \leq \frac{n}{L}.
    \end{equation*}
    and
    \begin{equation*}
      \sum_{n=0}^N\PP(E|X_1=n)\PP(X_1=n)
      \leq \frac{N}{L^2}.
    \end{equation*}
    By plugging this bound and the bound \eqref{eq:BE9}  into equation \eqref{eq:BE5} we get
    \begin{equation*}
      \Vert\gamma_L^N-\bar\gamma_L^N\otimes\bar\gamma_L^N\Vert
      \leq\PP(X_2\neq Y_2)
      \leq  \frac{2N}{L^2}.
    \end{equation*}
    The above estimate, together with \eqref{eq:BE1} concludes the proof.
\end{proof}

\section{Equilibrium properties of the nonlinear process}

In this section we study the long time behavior of the $\psi$-nonlinear process, corresponding to the GRBB process.
We will introduce some technical hypothesis on the nonlinear process which are satisfied in all the examples of Section~\ref{sec:app}.

We need some additional notation.
Given $\mu\in\mathcal{P}(\integer_+)$ we denote with $m_\mu$ the mean of $\mu$, with $\sigma^2_\mu$ its variance and with $\hat\mu$ its characteristic function.
 
The condition below is the analogue, for the nonlinear process, of equation \eqref{eq:co}, which assures the conservation of particles for the GRBB process.
\begin{con}
  \label{con:2}
  Assume that $\EE(B^q)=1-q(\{0\})$  for any $q\in\mathcal{P}(\integer_+)$ and that the map $\psi\colon\mathcal{P}(\integer_+)\to\mathcal{P}(\integer_+)$ depends only on $q(\{0\})$.   
\end{con}

\begin{re}
  If Condition~\ref{con:1} and equation \eqref{eq:co} hold, Condition~\ref{con:2} is equivalent to uniform integrability of the family $\{B_1^{L,q}\}_{L\in\natural}$.
  In fact Condition~\ref{con:1} implies that $B_1^{L,q}\Rightarrow B^q$ as $L\to+\infty$ and by equation \eqref{eq:co}:  $\EE(B_1^{L,q})=1-q(\{0\})$.
\end{re}

Particles conservation of the GRBB process gives conservation of the mean of the corresponding nonlinear process as explained in the following lemma.
\begin{lemma}
  \label{le:r}
  Assume Condition \ref{con:2} and $\EE(\eta(0))=r\in[0,+\infty]$.
  Then $\EE(\eta(t))=r$, $\forall t\geq0$.
\end{lemma}

\begin{proof}
  The proof is obtained by induction.
  Assume first that $\EE(\eta(t))=r<+\infty$.
  By equation \eqref{eq:nl1}:
  \begin{equation*}
    \EE\left(\eta(t+1)|\eta(t)\right)
    =\eta(t)-\ind(\eta(t)>0)+\PP(\eta(t)>0)
  \end{equation*}
  Then
    \begin{equation*}
      \EE\left(\eta(t+1)\right)=
    \EE\left[\EE\left(\eta(t+1)|\eta(t)\right)\right]
    =r.
  \end{equation*}
  When $r=+\infty$, again for equation \eqref{eq:nl1}, we have that $\eta(t+1)$ is obtained by adding a finite mean random variable to an infinite mean one.
\end{proof}

To study the long time behavior of the $\psi$-nonlinear process we introduce the following discrete time queue process.

\begin{de}
  Let $\mu\in\mathcal{P}(\integer_+)$.
  The G$_\mu$/D/1 queue $(\zeta(t))_{t\geq0}$ is the Markov chain with values in $\integer_+$ defined as follows.
  Assume that, for some $t\geq0$, $\zeta(t)=\zeta\in\integer_+$ then 
  \begin{equation}
  \label{eq:gd1}
  \zeta(t+1)=\zeta-\ind(\zeta>0)+B.
\end{equation}
$B$ is a random variable with distribution $\mu$.
It is independent from everything and independently generated at each time step $t$.
\end{de}

The long time behavior of the G$_\mu$/D/1 queue and its invariant probability measure are described in the theorem below.

\begin{teo}
  \label{teo:md1}
  The G$_\mu$/D/1 queue,  with $m_\mu<1$, is an aperiodic irreducible positive persistent Markov chain.
   Its invariant probability measure $\pi_\mu$ has characteristic function
  \begin{equation}
    \label{eq:cf1}
      \hat\pi_\mu(x)
      =\frac{(1-m_\mu)\hat\mu(x)(e^{i x}-1)}{e^{ix}-\hat{\mu}(x)},
      \qquad\qquad
      x\in\real,
  \end{equation}
  and
  \begin{equation}
    \label{eq:cf2}
    m_{\pi_\mu}
    =\frac{\sigma_\mu^2+m_\mu(1-m_\mu)}{2(1-m_\mu)}.
  \end{equation}
\end{teo}

    \begin{proof}
      By Markov inequality $1-\mu(\{0\})\leq m_\mu<1$ so that $\mu(\{0\})>0$ and the chain is aperiodic.
      The irreducibility and positive persistence follow directly by the dynamics of the G$_\mu$/D/1 queue.
      Let $\pi_\mu$ be the invariant probability measure of the G$_\mu$/D/1 queue $(\zeta(t))_{t\geq0}$, then by invariance and \eqref{eq:gd1}
      \begin{multline*}
        \hat\pi_\mu(x)
        =\sum_{\zeta}\pi_\mu(\{\zeta\}) e^{ix\zeta}
        =\sum_{\zeta}\pi_\mu(\{\zeta\})\EE_\zeta\big[e^{ix(\zeta(1))}\big]
        =\hat\mu(x)\sum_{\zeta}\pi_\mu\big(\{\zeta\}\big)e^{ix(\zeta-\ind(\zeta>0))}\\
        =\hat\mu(x)\big[\pi_\mu (\{0\})(1-e^{-ix})+\hat \pi_\mu (x) e^{-ix}\big].
      \end{multline*}
      Thus
      \begin{equation*}
              \hat \pi_\mu (x)
      =\frac{e^{i x}-1}{e^{ix}-\hat{\mu}(x)}\,\pi_\mu (\{0\})\hat\mu(x).
      \end{equation*}
      Taking the limit for $x\to0$ in the above equation we get
      \begin{equation*}
        1
        =\hat \pi_\mu (0)
        =\frac{\pi_\mu (\{0\})}{1-m_\mu},
      \end{equation*}
      so that $\pi_\mu$ has characteristic function given by \eqref{eq:cf1}.
      Equation \eqref{eq:cf2} follows from \eqref{eq:cf1} by computing $\hat\pi_\mu'(0)$.
    \end{proof}

The next lemma links the $\psi$-nonlinear process starting from $\pi_\mu$ with the G$_\mu$/D/1 queue.

\begin{lemma}
  Given $\mu\in\mathcal{P}(\integer_+)$ with $m_\mu<1$ consider a $\psi$-nonlinear process starting from $\pi_\mu$.
  If $\psi(\pi_\mu)=\mu$ then the $\psi$-nonlinear process is the G$_\mu$/D/1 queue.
\end{lemma}

\begin{proof}
  Assume that $\eta(t)\sim\pi_{\mu}$ for a $t\geq0$.
  Then \eqref{eq:nl1}, as $q=\pi_\mu$, holds with $B^q\sim\mu$.
  Thus equation \eqref{eq:nl1} defines the one step evolution of the G$_\mu$/D/1 queue with arrival distribution $\mu$and $\eta(t+1)\sim\pi_\mu$.
\end{proof}

Below we state a lemma which assures a uniform bound on exponential moments of the G$_\mu$/D/1 queue and will be used in the proof of Theorem~\ref{teo:st}.

\begin{lemma}
  \label{lem:ui}
    Let $(\zeta(t))_{t\geq0}$ be  the G$_\mu$/D/1 queue with $m_\mu<1$.
    Then there exist a positive constant $\lambda_\mu$, such that for
    any $\lambda\in[0,\lambda_\mu]$ there is a positive constant $C$, depending only on $\lambda$ and $\mu$ such that
    \begin{equation*}
      \EE_\zeta(e^{\lambda\zeta(t)})\leq
      C e^{\lambda\zeta},
    \end{equation*}
    for any $\zeta\in\integer_+$.
\end{lemma}

\begin{proof}
  As $m_\mu<1$, we can find $\lambda_\mu$ such that $e^{-\lambda}\hat\mu(-i\lambda)\in(0,1)$ for any $\lambda\in(0,\lambda_\mu]$.
  For  $\lambda\in(0,\lambda_\mu]$ define $f(\zeta):=e^{\lambda \zeta}$ and let $P$ be the transition matrix of the Markov chain $(\zeta(t))_{t\geq0}$.
  Define $\gamma:=1-e^{-\lambda}\hat\mu(-i\lambda)$ and $C:=\hat\mu(-i\lambda)(1-e^{-\lambda})$.
  Then
    \begin{equation*}
    Pf(\zeta)=
    \EE_\zeta(e^{\lambda\zeta(1)})=
    e^{\lambda(\zeta-\ind(\zeta>0))}\hat\mu(-i\lambda).
  \end{equation*}
  Thus
    \begin{equation*}
   Pf(\zeta)-f(\zeta)=
    \begin{cases}
          \hat\mu(-i\lambda)-1 & \text{if $\zeta=0$,}\\
              e^{\lambda\zeta}\big(e^{-\lambda}\hat\mu(-i\lambda)-1\big) & \text{if $\zeta>0$}
    \end{cases}
  \end{equation*}
  and
    \begin{equation}
    \label{eq:gc}
    Pf(\zeta)-f(\zeta)\leq
    -\gamma f(\zeta)+C.
  \end{equation}
  Iterating \eqref{eq:gc} we obtain
  \begin{equation*}
    P^tf(\zeta)\leq
    (1-\gamma)^tf(\zeta)+\frac{C}{\gamma}
    \qquad\qquad
    t\geq0,
  \end{equation*}
  and 
  \begin{equation*}
    \EE_\zeta(e^{\lambda\zeta(t)})
    \leq C e^{\lambda\zeta}.
  \end{equation*}
  The case $\lambda=0$ is trivial.
\end{proof}

The next technical condition is a thinning condition of the family $\{\mu^q\}_{ q\in\mathcal{P}(\integer_+)}$.
It holds for any of the applications of Section~\ref{sec:app}, see Remark~\ref{re:agg} below.

\begin{con}
  \label{con:3}
  For any $q\in\mathcal{P}(\integer_+)$ let $\mu^q$ be the distribution of $B^q$ and $X_1,X_2,\dots$ be independent Bernoulli random variables, independent from $B^q$.
  Then $\sum_{k=1}^{B^q}X_k$, where $\sum_{1}^0:=0$, has distribution $\mu^{q'}$ for some $q'\in\mathcal{P}(\integer_+)$.
\end{con}

Condition~\ref{con:3} is used in the next theorem to prove that the nonlinear process weakly converges to a unique stationary distribution.

\begin{teo}
  \label{teo:st}
  Assume that $\EE(\eta(0))=r\in[0,1)$, $\EE(e^{\lambda\eta(0)})<+\infty$ for some $\lambda>0$, Conditions \ref{con:2} and \ref{con:3} hold and that there exists a unique $\bar\pi\in\{\pi_{\mu^q}\}_{ q\in\mathcal{P}(\integer_+)}$ such that $m_{\bar\pi}=r$.
  Then $\eta(t)\Rightarrow\bar\pi$ as $t\to+\infty$.
\end{teo}

\begin{re}
  \label{re:agg}

  We briefly discuss the hypothesis and the consequences of the above theorem.
  The uniqueness assumption on $\bar\pi\in\{\pi_{\mu^q}\}_{ q\in\mathcal{P}(\integer_+)}$  with mean $r$ is used to identify the weak limit in the statement of the theorem.
  In the context of conservative particles systems the finite volume ergodic measures are usually parametrized by the mean occupation number.
  Thus it is rather natural to expect that the same holds for the infinite volume limit.

  For the applications of Section~\ref{sec:app} Condition~\ref{con:3} can be verified by computing and recognizing the characteristic function of the random sum $\sum_{k=1}^{B^q}X_k$ while the uniqueness assumption on $\bar\pi$ can be verified using \eqref{eq:cf2} and Condition~\ref{con:2}.
  We omit these computations for brevity.
    However, as $\bar\pi$ is the only measure with mean $r$ which has characteristic function given by \eqref{eq:cf1}, it is possible to recover its expression by using equations \eqref{eq:cf1}, \eqref{eq:cf2} and inverting the characteristic function.
  For the Fermi-Dirac case we obtain a Bernoulli measure, for the Maxwell-Boltzmann case the stationary measure of the M/D/1 queue (see Example~6.4 of \cite{Sh:Th:Gr:Ha}) while for the Bose-Einstein case the geometric distribution supported on $\integer_+$.

\end{re}

\begin{proofof}{Theorem \ref{teo:st}}
  We first observe that by Lemma~\ref{le:r} we have $\EE(\eta(t))=r$ for any $t\geq0$ and this, via Markov inequality, implies the tightness of the sequence of distributions of $(\eta(t))_{t\geq0}$.
  Furthermore, denoting with $q(t)$ the distribution of $\eta(t)$, by equation \eqref{eq:nl1} we have, for any $x\in\real$,
    \begin{equation*}
      \begin{split}
            \EE\left(e^{ix\eta(t+1)}\right)&=
    \EE\left[\EE\left(e^{ix\eta(t+1)}|\eta(t)\right)\right]
    =\EE\left[e^{ix(\eta(t)-\ind(\eta(t)>0))}\EE\left(e^{ix B^{q(t)}}|\eta(t)\right)\right]\\
       &=\hat\mu^{q(t)}(x)\EE\left(e^{ix(\eta(t)-\ind(\eta(t)>0))}\right).
      \end{split}
  \end{equation*}
  By tightness we can choose a subsequence $(\eta(\bar t))_{\bar t\geq0}$ of $(\eta(t))_{t\geq0}$ with weak limit point $\bar\eta$.
  Taking the limit, for $\bar t\to+\infty$, in the previous equation we get:
  \begin{equation}
    \label{eq:cf}
    \EE\left(e^{ix\bar\eta}\right)=
     \hat\mu^{\bar q}(x)\EE\left(e^{ix(\bar\eta-\ind(\bar\eta>0))}\right),
  \end{equation}
  where $\bar q$ is the distribution of $\bar\eta$.
  Observe that
  \begin{equation*}
    \begin{split}
          \EE\left(e^{ix (\bar\eta-\ind(\bar\eta>0))}\right)&=
    \EE\left(e^{ix \bar\eta},\bar\eta=0\right)+ e^{-ix}\EE\left(e^{ix \bar\eta},\bar\eta>0\right)\\
    &=\bar q(0)+e^{-ix}\left[\EE\left(e^{ix \bar\eta}\right)-\bar q(0)\right].
    \end{split}
  \end{equation*}
  Plugging this expression in the right hand side of equation \eqref{eq:cf} and solving it we get
  \begin{equation*}
    \EE\left(e^{ix \bar\eta}\right)
    =\frac{\bar q(0) (e^{ix}-1) \hat\mu^{\bar q}(x)}{e^{ix}-\hat\mu^{\bar q}(x)}.
  \end{equation*}
  Taking the limit for $x\to0$ in the above equation we obtain that $\bar q(0) =1-m_{\mu^{\bar q}}$.
  Thus,  by Theorem~\ref{teo:md1}, the limit points of the distributions of $(\eta(t))_{t\geq0}$ belong to
  \begin{equation*}
    \{\pi_{\mu^q}\}_{ q\in\mathcal{P}(\integer_+)}.
  \end{equation*}
  As, by hypothesis, there is only one element in $\{\pi_{\mu^q}\}_{ q\in\mathcal{P}(\integer_+)}$ with mean $r$,  
  to prove the uniqueness  of the limit it is enough to show that
  \begin{equation}
    \label{eq:ui}
    \EE(\bar \eta) =\lim_{\bar t\to+\infty}\EE(\eta(\bar t))=r,
  \end{equation}
  by proving uniform integrability of the sequence $(\eta(t))_{t\geq0}$ (see for example Theorem 25.11 of \cite{Bi}).
  In fact the nonlinear process $(\eta(t))_{t\geq0}$ can be coupled with a G$_{\mu^{q_r}}$/D/1 queue $(\zeta(t))_{t\geq0}$ with $q_r(\{0\})=1-r$, so that $\PP(\eta(t)\leq \zeta(t))=1$ for any $t\geq0$, and uniform integrability of $(\eta(t))_{t\geq0}$ will follow by uniform integrability of $(\zeta(t))_{t\geq0}$.
  Observe that
  \begin{equation*}
    \PP(\eta(t)>0)\leq
    \EE(\eta(t))=r.
  \end{equation*}
  Take a sequence of \iid $\{B_t^{q_r}\}_{t\geq0}$, distributed accordingly with $\mu^{q_r}$ and for any $t>0$ take and a sequence of \iid Bernoulli random variables $Y_{1\, t},Y_{2\, t},\dots$ with parameter $\PP(\eta(t)>0)/r$ such that sequences with different $t$ are independent and independent from each $B_t^{q_r}$.
 Now define
 \begin{equation*}
   B_t:=
   \sum_{k=1}^{B_t^{q_r}}Y_{k\,t}.
 \end{equation*}
 Then $\{B_t\}_{t\geq0}$ are independent random variables and by Condition~\ref{con:3} have distribution belonging to $\{\mu^q\}_{q\in\mathcal{P}(\integer_+)}$.
 By Condition~\ref{con:2} and Condition~\ref{con:3} $B_t\sim\mu^{q(t)}$ because $\EE(B_t)=\PP(\eta(t)>0)$.
 This implies $B_t\leq B_t^{q_r}$ \as for any $t\geq0$ so that if $\zeta(0)=\eta(0)$ and define
 \begin{align*}
   \eta(t+1)&:=\eta(t)-\ind(\eta(t)>0)+B_t\\
   \zeta(t+1)&:=\zeta(t)-\ind(\zeta(t)>0)+B_t^{q_r}
 \end{align*}
 we have that $\eta(t)\leq\zeta(t)$ \as for any $t>0$.
 To obtain uniform integrability of the nonlinear process observe that by Lemma~\ref{lem:ui},  taking $\lambda>0$ small enough,
 \begin{equation*}
   \EE(e^{\lambda\eta(t)})
   \leq\sum_\eta\PP(\eta(0)=\eta) \EE_\eta(e^{\lambda\zeta(t)})
   \leq C_r \EE(e^{\lambda\eta(0)})<+\infty.
 \end{equation*}
\end{proofof}

\end{document}